\documentclass[a4paper,11pt]{article}

\usepackage{hyperref}

\usepackage{latexsym}
\usepackage{amsopn,amscd}
\usepackage{amsfonts}
\usepackage{amssymb}
\usepackage{amsthm}
\usepackage{amsmath}
\usepackage{enumerate}

\usepackage{url}
\usepackage[all]{xy}

\newtheorem{thm}{Theorem}[section]
\newtheorem{cor}[thm]{Corollary}
\newtheorem{lem}[thm]{Lemma}
\newtheorem{prop}[thm]{Proposition}

\theoremstyle{definition}

\theoremstyle{remark}
\newtheorem{rema}[thm]{Remark}

\title
{K\"ahler structures on $\T^* \GG$ having as underlying symplectic form the standard one}
\author{
J.~Huebschmann and K. Leicht
\\[0.3cm]
 USTL, UFR de Math\'ematiques\\
CNRS-UMR 8524
\\
Labex CEMPI (ANR-11-LABX-0007-01)
\\
59655 Villeneuve d'Ascq Cedex, France\\
Johannes.Huebschmann@math.univ-lille1.fr
 }

\numberwithin{equation}{section}

\def\prin{\xi}

\def\form{\mathcal A}

\def\GG{G}
\def\gg{\mathfrak g}
\def\ggg{\overline{\mathfrak g}}
\def\bra{[\,\cdot\, , \,\cdot\, ]}

\def\oo{\omega}
\def\ppsi{c}
\def\llambda{s}
\def\qqq{a}
\def\ttt{\mathfrak t}

\def\PGG{P}
\def\rg{g}
\def\bbb{\beta}
\newcommand{\T}{\mathrm{T}}

\newcommand{\mathscr}{\mathcal}

\long
\def\MSC#1\EndMSC{\def\arg{#1}\ifx\arg\empty\relax\else
      {\par\narrower\noindent
      2010 Mathematics Subject Classification: #1\par}\fi}
\long
\def\KEY#1\EndKEY{\def\arg{#1}\ifx\arg\empty\relax\else
    {\par\narrower\noindent
      Keywords and Phrases: #1\par}\fi}

\begin{document}

\maketitle
\begin{abstract} For a connected Lie group $G$,
we show that a complex structure
on the total space $\T G$ of the tangent bundle
of $G$ that
is left invariant and has the property
that
each left translation $\GG$-orbit is a totally real
submanifold
is induced from a smooth immersion of $\T G$ into
the complexification $G^{\mathbb C}$ of $G$.
For $G$
compact and connected,
we then characterize
left invariant and biinvariant complex
structures on the total space $\T^*\GG$
of the cotangent bundle of $\GG$
which combine with the 
tautological symplectic structure
to a K\"ahler structure.
\end{abstract}

\MSC 53C55 (32Q15 53D20)
\EndMSC

\KEY Invariant K\"ahler structure, cotangent bundle, Lie group
\EndKEY

\tableofcontents

\section{Introduction}

In \cite{MR2472033}, the first named author has developed,
in collaboration with two physicists, 
a gauge model for quantum mechanics on a stratified space.
The underlying unreduced phase space is the total space
$\T^*\GG$ of the cotangent bundle of a compact connected Lie group
$\GG$,  endowed with the tautological symplectic structure,
and the reduced phase space
is the singular symplectic quotient of 
$\T^*\GG$ with respect to conjugation.
The standard identification of 
$\T^*\GG$ with the complexification $\GG^{\mathbb C}$
of $\GG$ via a choice of invariant inner product on the Lie algebra
$\gg$ of $\GG$ and the standard polar decomposition map
from $\T \GG \cong \GG \times \gg$ to
$\GG^{\mathbb C}$ turns
$\T^*\GG$ into a K\"ahler manifold, the K\"ahler structure
being $\GG$-biinvariant.
We refer to this structure as the {\em standard structure\/}.
At the reduced level, the gauge model 
in \cite{MR2472033} is built on the associated 
singular K\"ahler quotient.
The complex structure on
$\T^*\GG$ resulting from the identification
with $\GG^{\mathbb C}$ has no interpretation in physics,
and the question arises as to what extent 
the physical interpretation depends on the choice of
complex structure. To attack this question,
as a preliminary step,
in the present paper,
we classify all left invariant and all biinvariant complex
structures on $\T^*\GG$
which combine with the 
tautological symplectic structure
to a K\"ahler structure.
To this end, we elaborate on an approach in
\cite{MR1989647} aimed at describing
 K\"ahler structures on a space of the kind $\T^*\GG$ 
and at exploring their Ricci curvatures.
In a sense, we globalize some of the results
in \cite{MR1989647}.
More precisely, we show that,
given a connection 1-form and a horizontal 1-form
as in Proposition 3.1 of  \cite{MR1989647},
when these forms satisfy the integrability conditions
spelled out in that Proposition and hence determine
a complex structure on $P=G\times \gg\cong \T G$,
this complex structure on $P$ is actually induced from a
smooth immersion $P \to \GG^{\mathbb C}$.
A precise statement is given as Theorem \ref{fundth24} below.
In Theorem \ref{fundth241}
we will, furthermore, give a criterion which characterizes those
complex structures $J$ on $P$ which are $G$-biinvariant.
In Theorem \ref{kahlertco},
for $G$ compact and connected,
we then single out those complex structures on $\T ^* \GG$
which combine with the tautological symplectic structure to
a K\"ahler structure.
In Subsection \ref{nontrivial}, we 
illustrate our approach with a class of examples
more general than the standard structure
on $\T ^* \GG$.

This paper is based on the second-named author's doctoral dissertation
to be submitted in partial fulfillment of the requirements for the 
PhD degree at the university Lille 1. 

\section {Gauge theory with structure group acting from the left}
\label{gaugeright}
In the standard setup, cf. e.~g.  \cite{MR0152974},
the structure group acts on the total space 
of a principal bundle from the right.
Below we will work with  principal bundles 
having structure group
acting from the left.
For ease of exposition, and to introduce notation, we briefly 
explain the requisite formalism.

Let $\GG$ be a Lie group and $\gg$  its Lie algebra of left invariant
vector fields, the Lie bracket being written as $\bra$.
We denote the Lie algebra of
right invariant vector fields on $\GG$
by $\ggg$,
with Lie bracket 
$\bra^{-} 
\colon 
\ggg \times \ggg \to
\ggg$.
When we identify $\gg$ with $\ggg$ as vector spaces
via the canonical identifications with the tangent space $\T_e\GG$
to $\GG$ at the identity element $e$ of $\GG$,
the bracket $\bra^{-}$ gets identified with the negative of $\bra$.

Given a smooth manifold $M$ and a vector space $V$,
let $\mathcal A(M, V)$
denote the graded vector space of
$V$-valued differential forms
on $M$. We denote the de Rham operator by $d$.
Let $\prin \colon P \to M$ be a principal $\GG$-bundle
having the structure group $\GG$ acting on $P$ from the left,
let $V$ be a right $\GG$-module,
write $\mathcal A_{\mathrm{basic}}(P, V)\subseteq  \mathcal A(P, V) $
for the graded vector space of basic $V$-valued
differential forms on $P$ and, with an abuse of notation,
write the induced infinitesimal
$\ggg$-action on $V$ 
 from (beware) the left
as
$\bra^{-} 
\colon 
\ggg \times V \to
V$.
This infinitesimal action induces
the pairing
\begin{equation}
[\, \cdot\, ,\, \cdot\, ]^{-}\colon
\mathcal A(P, \ggg) \times \mathcal A(P, V)
\longrightarrow
\mathcal A(P, V).
\label{rpair}
\end{equation}

The pairing \eqref{rpair}
is, in particular, defined on $\mathcal A(\GG, \ggg)$, and
the right invariant Maurer-Cartan form 
$\overline \omega_{\GG}\colon  \T \GG \to \ggg$
of $\GG$ 
satisfies
the {\em Maurer-Cartan equation\/}
or {\em structure equation\/} 
\begin{equation}
d\overline \omega_\GG+
\tfrac 12 [\overline \omega_\GG,\overline \omega_\GG]^{-}=0. 
\label{MCr2}
\end{equation}
A {\em connection form\/} for $\prin$
is a $\ggg$-valued 1-form $\theta\colon \T P \to \ggg$
which, on the vertical part of $\T P$, restricts to the obvious
extension of the  right invariant Maurer-Cartan form 
$\overline \omega_{\GG}$ and which is $\GG$-equivariant in the sense that
\[
\theta (xY)=\mathrm{Ad}_x(Y),
\]
for any tangent vector $Y$ to $P$ and $x \in \GG$.
The {\em curvature\/} of $\theta$ is then given by the (familiar) expression
\begin{equation}
d\theta +
\tfrac 12 [\theta,\theta]^{-} \in \mathcal A^2(P, \ggg),
\label{rcurva}
\end{equation} 
necessarily a basic  $\ggg$-valued 2-form on $P$.
On the basic forms $\mathcal A_{\mathrm{basic}}(P, V)$, 
the operator
$d^\theta$ 
of covariant derivative  
is given by
\begin{equation}
\label{cova} 
d^\theta = d + [\theta, \, \cdot\,]^{-}
\colon 
\mathcal A_{\mathrm{basic}}(P, V)
\longrightarrow
\mathcal A_{\mathrm{basic}}(P, V),
\end{equation}
where the notation $[\, \cdot\, ,\, \cdot\, ]^{-}$ 
is slightly abused.
Notice that the values of the sum
$d + [\theta, \, \cdot\,]^{-}$ 
(restricted to $\mathcal A_{\mathrm{basic}}(P, V)$)
lie
in $\mathcal A_{\mathrm{basic}}(P, V)$
but not necessarily the values of the individual
operators $d$ or 
$ [\theta, \, \cdot\,]^{-}$.

\section{Invariant complex structures on the total space of the tangent bundle of a Lie group}

\subsection{Left invariant complex structures}
\label{lefti}
We write the tangent bundle 
of a smooth manifold $M$
as
$\tau_M\colon \T M \to M$.  
Let $G$ be a connected Lie group.
Henceforth the terms {\em left translation\/}
and {\em right translation\/}
mean left translation
and right translation, respectively, with respect to
members of $G$.
We will say that a complex structure $J$ on the total space
$\T G$ of the tangent bundle of $G$
is {\em admissible\/}
when $J$ is left invariant and when
each left translation $\GG$-orbit is a totally real
submanifold.
Our aim is to explore such admissible complex structures.

Let 
$G^{\mathbb C}$ be the complexification of $G$ 
\cite{MR0206141};
we denote the image of $G$ in $G^{\mathbb C}$ by $\overline G$.
When $G$ is compact---our main case of interest---the 
canonical homomorphism from $G$ to  $G^{\mathbb C}$ is injective, 
and we can identify $G$ with $\overline G$.
For general $G$, a left translation equivariant immersion of
$\T \GG$ 
into $G^{\mathbb C}$, necessarily onto an open subset,
plainly induces a left translation invariant complex structure
on $\T \GG$, and 
when the immersion is also right translation equivariant,
the complex structure on $\T \GG$ is biinvariant.
Theorem \ref{fundth24} below says that
any left translation invariant complex structure on $\T \GG$
arises in this manner.

For the rest of the paper,
it will be convenient to trivialize the tangent bundle of $G$
and to play down the linear structure of the fibers.
When we view 
the Lie algebra
$\gg$ 
(of left-invariant vector fields on $G$)
merely as an affine manifold, we write it as
$\mathbb A_{\gg}$. Left translation yields the familiar 
$\GG$-equivariant diffeomorphism
\begin{equation}
\GG \times \mathbb A_{\gg} \longrightarrow 
\T \GG,
\label{leftt}
\end{equation}
the left $\GG$-action on the factor $\mathbb A_{\gg}$
being trivial,
and we will exclusively work with 
$\GG \times \mathbb A_{\gg}$ (rather than with $\T \GG$).
Accordingly we will say that a 
complex structure $J$ on
$\GG \times \mathbb A_{\gg}$ is {\em admissible\/}
when $J$ is left $G$-invariant and when
each left translation $\GG$-orbit is a totally real
submanifold; we will then
say that $J$ is an {\em admissible almost
complex structure\/} when $J$ is not required to
satisfy the integrability condition.
Notice that \eqref{leftt} is also $\GG$-equivariant
relative to right translation when the $\GG$-action
on $\GG \times \mathbb A_{\gg}$ from the right
is given by right translation in $\GG$ and the adjoint action on 
$\mathbb A_{\gg}$. Notice also that there is an obvious
bijective correspondence
between admissible (almost) complex structures on $\T \GG$ and
on $\GG \times \mathbb A_{\gg}$.

Let $\gamma\colon \mathbb A_{\gg}\to \GG^{\mathbb C}
$ 
be a smooth map
having the property that the
 composite
\begin{equation}
\begin{CD}
 \mathbb A_{\gg}
@>{\gamma}>>
 G^{\mathbb C}
@>{\pi}>>
\overline G\big\backslash  G^{\mathbb C}
\end{CD}
\label{compo}
\end{equation}
is a smooth map of maximal rank;
the domain and range of $\gamma$ being smooth manifolds of the same dimension,
$\gamma$ is necessarily an immersion and, furthermore, a submersion onto
an open subset of $\overline G\big\backslash  G^{\mathbb C}$.
Then
 the map
\begin{equation}
\Pi_{\gamma}\colon \GG \times \mathbb A_{\gg} \longrightarrow 
\GG^{\mathbb C} ,
\ (x,\qqq)\longmapsto x \,\gamma(\qqq),
\label{genpolar}
\end{equation}
is a left translation equivariant 
smooth immersion onto an open subset of $\GG^{\mathbb C}$.
Hence the complex structure of $\GG^{\mathbb C}$
induces an admissible complex structure $J_{\gamma}$ on 
$\GG \times\mathbb A_{\gg}$.
We will refer to \eqref{genpolar} as the {\em generalized polar map\/}
associated to $\gamma$.
Notice when $\gamma$ is equivariant with respect to the
adjoint action and $\GG$-conjugation in $\GG^{\mathbb C}$,
the complex structure $J_{\gamma}$ is
also right translation invariant 
and hence biinvariant.

\begin{thm}\label{fundth24}
Let $J$ 
be an admissible
complex structure on $\GG \times \mathbb A_{\gg}$.
There is a smooth map 
$\gamma_J\colon \mathbb A_{\gg}\to  \GG^{\mathbb C}$,
unique up to right multiplication by a constant member of $\GG^{\mathbb C}$,
such that the associated generalized polar map
{\rm \eqref{genpolar}}
is a left translation equivariant 
holomorphic immersion onto an open subset of 
$\GG^{\mathbb C}$, and the composite $\pi \circ \gamma \colon 
 \mathbb A_{\gg}\to\overline G\big\backslash  G^{\mathbb C}$
is necessarily an immersion.
In particular, when $\gamma_J$ is $\GG$-equivariant
relative to the adjoint action and $\GG$-conjugation in $\GG^{\mathbb C}$,
the map $\Pi_{\gamma_J}$ is also right translation equivariant,
and the complex structure $J$
is then right invariant as well and hence biinvariant.
\end{thm}

Under the circumstances of
Theorem \ref{fundth24}, we will say that
$\gamma_J$ is an {\em admissible map inducing\/} $J$.
Notice we do not assert that 
a biinvariant admissible complex 
structure $J$
on $\GG \times \mathbb A_{\gg}$
is induced from a $\GG$-equivariant map 
$\gamma_J\colon \mathbb A_{\gg}\to  \GG^{\mathbb C}$.
 In the next subsection we shall explain
how a general biinvariant complex structure arises.

\subsection{Biinvariant complex structures}
\label{bcs}
The description of biinvariant complex structures on 
$\GG \times \mathbb A_{\gg}\cong \T \GG$
is more subtle. To prepare for it, 
let $G$ be a group, $H\subseteq G$ a subgroup,
and $B$ a simply connected (left) $H$-manifold.
The group $H$ acts on 
$\mathrm{Map}(B,G)$
by the association
\[
H \times \mathrm{Map}(B,G) \longrightarrow
\mathrm{Map}(B,G),\ (x,\gamma)\longmapsto {}^x\gamma,
\]
given by the explicit expression
\[
{}^x\gamma\colon B \longrightarrow G,\ 
{}^x\gamma(b) =x\gamma(x^{-1}b)x^{-1},\ x \in H,\ \gamma\colon B \to G.
\]
Thus a smooth map $\gamma\colon B \to G$ is
$H$-equivariant if and only if $\gamma$ is fixed under the
$H$-action on $\mathrm{Map}(B,G)$.
We will say that 
 a smooth map $\gamma\colon B \to G$ is {\em quasi\/} $H$-{\em equivariant\/}
when there is a smooth map $c\colon H \to G$ such that
\begin{equation}
 \gamma^{-1}(b)\,{}^x\gamma (b)=c(x),\ x \in H,\ b \in B.
\end{equation}
Thus  a smooth map $\gamma\colon B \to G$ is $H$-equivariant
if and only if it is quasi $H$-equivariant relative to the 
constant smooth map $c\colon H \to G$ where $c(x)=e$ as $x$ ranges over $H$.
Accordingly, we will say that a smooth left equivariant map
$\Pi\colon \GG \times \mathbb A_{\gg} \to \GG^{\mathbb C}$ 
is {\em quasi right\/}
$\GG$-equivariant when
there is a smooth map 
$c\colon \GG \to \GG^{\mathbb C}$ such that
\begin{equation}
\Pi((x,Y)z)
=\Pi(x,Y)c(z) z,\ x \in \GG,\ Y \in \mathbb A_{\gg},\ z \in \GG.
\end{equation}
In particular, a smooth left equivariant map
$\Pi\colon \GG \times \mathbb A_{\gg} \to \GG^{\mathbb C}$ is 
biinvariant if and only if it is quasi right
$\GG$-equivariant 
 relative to the 
constant smooth map $c\colon \GG \to \GG^{\mathbb C}$ where $c(x)=e$ as 
$x$ ranges over $\GG$.

\begin{thm}\label{fundth241}
Under the circumstances of Theorem {\rm \ref{fundth24}},
the left invariant complex structure $J$ on $\GG \times \mathbb A_{\gg}$
is biinvariant if and only if the smooth map
$\gamma_J\colon \mathbb A_{\gg} \to \GG^{\mathbb C}$
is, furthermore, quasi $\GG$-equivariant.
\end{thm}

\subsection{The standard structure}
\label{ordinarypo}
An example of a generalized polar map 
is the ordinary polar map.
With the notation
\begin{equation}
\gamma_{\mathrm{st}}\colon \mathbb A_{\gg}\longrightarrow \GG^{\mathbb C},
\quad
\gamma_{\mathrm{st}}(\qqq)=\mathrm{exp}(i\qqq),\  \qqq \in \mathbb A_{\gg},
\end{equation}
the ordinary polar map takes the form
\begin{equation}
\Pi=\Pi_{\mathrm{st}}\colon 
\GG \times \mathbb A_{\gg} \longrightarrow 
\GG^{\mathbb C} ,
\ (x,\qqq)\longmapsto x\gamma_{\mathrm{st}}(\qqq),
\ x\in G,\  \qqq \in \mathbb A_{\gg}.
\label{polardec}
\end{equation}
For $G$ compact, this map is a diffeomorphism
and thus plainly induces an admissible complex structure on 
$\GG \times \mathbb A_{\gg}$; we  refer to this structure as the
{\em standard complex structure\/} on
$\GG \times \mathbb A_{\gg}$ and denote it by $J_{\mathrm{st}}$.

For general $\GG$, in view of the classical expression for the derivative of
the exponential mapping,
cf. e.~g. \cite {MR754767} (II.1.7),
in terms of the 
$\gg^{\mathbb C}$-valued 
1-form
\[
\phi_{\mathrm{st}}= 
(d\gamma_{\mathrm{st}})\gamma_{\mathrm{st}}^{-1}
\colon \T \mathbb A_{\gg} \to \gg \oplus i \gg,
\]
at $\qqq \in \mathbb A_{\gg}$,
the derivative
\begin{equation*}
(d\phi_{\mathrm{st}})_{\qqq}
\colon
\T_{\qqq} \mathbb A_{\gg}\cong\mathfrak g
\to
\mathrm T_{\mathrm{exp}(i\qqq)}\GG^{\mathbb C}
\to
\T_e\GG^{\mathbb C} = \gg \oplus i \gg
\end{equation*}
is given by the association
\begin{equation}
V \longmapsto \frac{\mathrm{cos}(\mathrm{ad}(\qqq))-\mathrm{Id}}
{\mathrm{ad}(\qqq)} (V)+ i\frac{\mathrm{sin}(\mathrm{ad}(\qqq))}
{\mathrm{ad}(\qqq)} (V),\ V \in \gg.
\label{associ1}
\end{equation}

When $G$ is not compact, in general, the canonical map
$G\to G^{\mathbb C}$ is not injective, and the projection
$G \to \overline G$ to the image  $\overline G$
of $G$ in $G^{\mathbb C}$
is a covering projection.
Moreover, even when 
the complexification map $G \to G^{\mathbb C}$
is injective (so that 
$G \to \overline G$ identifies the two groups),
the topology of
$\overline G \big\backslash G^{\mathbb C}$ is in general non-trivial.
This happens, for example, when $G=\mathrm{SL}(2,\mathbb R)$.
Furthermore, in general, the ordinary polar map 
$\Pi_{\mathrm{st}}$ is not even 
a local diffeomorphism:
Indeed, in view of \eqref{associ1}, 
at $\qqq \in \mathbb A_{\gg}$,
the derivative of
\eqref{compo}
is given by the association
\begin{equation}
\T \mathbb A_{\gg}\cong \gg \longrightarrow \gg,
\ 
V \longmapsto\frac{\mathrm{sin}(\mathrm{ad}(\qqq))}
{\mathrm{ad}(\qqq)} (V),\ V \in \gg.
\end{equation}
Hence \eqref{compo}
is a local diffeomorphism 
at $\qqq \in \mathbb A_{\gg}$
if and only if
the linear endomorphism
$\frac{\mathrm{sin}(\mathrm{ad}(\qqq))}
{\mathrm{ad}(\qqq)}$
of $\gg$ is invertible.
This explains why, for non-compact semisimple $\GG$,
the adapted complex structure
is not defined on all of $\T \GG$;
see e.~g. \cite{MR2127990} for details.

\subsection{The infinitesimal version}

We view the left invariant Maurer-Cartan form 
 of $G$ as a
trivializable principal left $G$-bundle 
$\omega_G\colon \T \GG \to \mathbb A_{\mathfrak g}$.
Via the trivialization \eqref{leftt},
this bundle comes down to the trivial principal
left $\GG$-bundle
$\mathrm{pr_{\mathbb A_{\gg}}}\colon \GG \times\mathbb A_{\mathfrak g} \to \mathbb A_{\mathfrak g}$.
For better readability, in the present subsection,
we will write the total space 
$\GG \times\mathbb A_{\mathfrak g}$ as 
$\PGG$.
We denote 
the resulting foliation of $\PGG$ given by the fibers of $\omega_G$
by $\mathcal F_G$, and we write the tangent bundle of  $\mathcal F_G$ as
$\tau_{\mathcal F_G}\colon \T \mathcal F_G \to \PGG$. The vector bundle
$\tau_{\mathcal F_G}$ is the vertical subbundle of the tangent bundle
$\tau_{\PGG}$ of $\PGG$ with respect to principal left $G$-structure of $\PGG$.

The fundamental vector field map
$
\overline {\mathfrak g}\times \PGG \longrightarrow  \T \PGG
$
identifies the trivial vector bundle
${\mathrm{pr}_{\PGG}\colon\ggg\times \PGG
\to \PGG}
$ 
on $\PGG$
with the vertical subbundle ${\tau_{\mathcal F_G}\colon \T \mathcal F_G \to  \PGG}$ 
of  $\tau_{ \PGG}\colon \T  \PGG \to  \PGG$.
Likewise, the obvious map
is a diffeomorphism 
\[
\GG \times \T \mathbb A_{\gg} \to \PGG \times _{\mathbb A_{\gg}}  \T \mathbb A_{\gg}. 
\]
Consequently
the fundamental exact vector bundle sequence  
associated to the principal left $G$-bundle 
$\mathrm{pr_{\mathbb A_{\gg}}}\colon \PGG \to \mathbb A_{\gg}$
takes the form
\begin{equation}
\begin{CD}
0
@>>>
\overline {\mathfrak g}\times  \PGG
@>>>
\T \PGG
@>>>
\GG \times \T \mathbb A_{\gg}
@>>>
0.
\end{CD}
\label{funda22}
\end{equation}
The projection 
$\mathrm{pr_{\mathbb A_{\gg}}}$
induces an isomorphism
\begin{equation}
\mathcal A(\mathbb A_{\gg},\ggg) \longrightarrow 
\mathcal A_{\mathrm{basic}}(\PGG,\ggg)
=\mathcal A_{\mathrm{basic}}(\GG \times \mathbb A_{\gg},\ggg)
\label{immediate}
\end{equation}
of graded vector spaces onto the graded vector space of basic forms.
We will say that 
a basic 
$\ggg$-valued
1-form 
on $ \PGG=\GG \times \mathbb A_{\gg}$ is {\em regular\/}
when the associated 
$\ggg$-valued
1-form on  $\mathbb A_{\gg}$ has maximal rank,
that is, is an isomorphism 
${\T _\qqq \mathbb A_{\gg} \to \ggg}$
of real vector spaces for every point $\qqq$ of $\mathbb A_{\gg}$.
The following observation is essentially Proposition 3.1 in 
\cite {MR1989647}, and we refer to that paper for a proof.
We reproduce the wording here to introduce notation.

\begin{prop} \label{fundth1}
{\rm (i)}
Let $J$ be an admissible 
almost complex structure 
on $ \PGG=\GG \times \mathbb A_{\gg}$ 
and let
$
\tau_{ \PGG}=\tau_{\mathcal F_G} 
\oplus J(\tau_{\mathcal F_G})
$
be the associated  Whitney sum decomposition
of the tangent bundle ${\tau_{ \PGG}\colon \T \PGG \to \PGG }$
of $ \PGG$ as a sum 
of the vertical subbundle $\tau_{\mathcal F_G}$ and
$J(\tau_{\mathcal F_G})\colon J(\T \mathcal F_G)\to  \PGG $.
Then the
projection
\[
\T  \PGG= \T \mathcal F_G \oplus J(\T \mathcal F_G) \longrightarrow
 \T \mathcal F_G,
\]
combined with the projection
$
\T \mathcal F_G \cong\ggg \times  \PGG \longrightarrow \ggg,
$
yields a principal left $G$-connection form 
$\theta_J\colon \T  \PGG \to\ggg$
for $\mathrm{pr_{\mathbb A_{\gg}}}$,
and the $\ggg$-valued $1$-form
\begin{equation}
L_J= -\theta_J \circ J\colon \T  \PGG \to\ggg
\label{L}
\end{equation}
is basic and regular.

\noindent
{\rm (ii)}
Conversely, a left $G$-connection form 
$\theta\colon \T  \PGG \to \ggg$ for 
$\mathrm{pr_{\mathbb A_{\gg}}}$ and a basic regular
$\ggg$-valued
$1$-form
$
L\colon \T  \PGG \to\ggg
$
determine a unique admissible
almost complex structure $J$ on $ \PGG=\GG \times \mathbb A_{\gg}$
such that $\theta=\theta_J$ and $L=L_J$.

\noindent
{\rm (iii)}
Under the circumstances of 
$(i)$ or $(ii)$ above,
the almost complex structure $J$ is integrable if and only if
\begin{align}
F_{\theta_J}&=L_{J} \wedge^{-} L_{J}
\,(=\tfrac 12[L_{J}, L_{J}]^{-})
\in 
\mathcal A^2(\T  \PGG,\ggg),
\label{integra1}
\\
d^{\theta_J}L_J&=0 .
\label{integra2}
\end{align}
 
\qed \end{prop}

\subsection{Interpretation in terms of the complexified Lie algebra}
\label{interp1}
The bundle $\mathrm{pr}_{\mathbb A_{\gg}}\colon \GG \times \mathbb A_{\gg} \to  \mathbb A_{\gg}$ being trivial,
the associated (trivial) principal left $G$-connection form
$\theta_{\GG}\colon \T (\GG \times \mathbb A_{\gg}) \longrightarrow \ggg$
is  the obvious extension
of the right invariant Maurer-Cartan form 
$\overline \omega_{\GG}\colon  \T \GG \to \ggg$
of $\GG$ to a principal left $G$-connection form
on $ \GG \times \mathbb A_{{\gg}}  $.

Let $J$ be an admissible 
almost complex structure on $\GG \times \mathbb A_{\gg}$.
Let $(\theta,L)$ be the pair of $\ggg$-valued 1-forms
on $\GG \times \mathbb A_{\gg}$ associated to $J$ by the construction in Proposition \ref{fundth1}.
Let
$\ppsi_{\theta}\colon \T \mathbb A_{\gg} \to \ggg$
denote the $\ggg$-valued 1-form
on $\mathbb A_{\gg}$
whose extension 
$\ppsi^G_{\theta}\colon G\times \T \mathbb A_{\gg} \to \ggg$
to $G\times  \T \mathbb A_{\gg}$
yields, via the projection from
$\T (\GG\times \mathbb A_{\gg})$
to
$\GG\times  \T \mathbb A_{\gg}$, a uniquely determined basic form
$\ppsi^G_{\theta}$ on $\GG\times \mathbb A_{\gg}$
so that
\begin{equation}
\theta=\theta_{\GG} + \ppsi^G_{\theta}.
\label{psitheta}
\end{equation}
Likewise
let
$\llambda_L\colon  \T \mathbb A_{\gg} \to \ggg$
be the  
$\ggg$-valued 1-form on $\mathbb A_{\gg}$ 
of maximal rank
associated to the basic 1-form $L$ through the identification \eqref{immediate}.

\begin{lem} \label{fundth11} {\rm (i)}
Let $J$ be an admissible almost complex structure 
on $\GG \times \mathbb A_{\gg}$,
let $(\theta,L)$ be the pair of $\ggg$-valued $1$-forms
on $\GG \times \mathbb A_{\gg}$ associated to $J$ by the construction 
in Proposition {\rm \ref{fundth1}}, and let
$(\ppsi_{\theta},\llambda_L)$ be the associated 
pair of $1$-forms  on $\mathbb A_{\gg}$.
The $\ggg^{\mathbb C}$-valued $1$-form 
\begin{equation}
\phi_{J}=\ppsi_{\theta} + i \llambda_L\colon  \T \mathbb A_{\gg} \longrightarrow 
\ggg^{\mathbb C} = \ggg\oplus i \ggg
\label{thetaL}
\end{equation}
on $\mathbb A_{\gg}$
has the property that its imaginary part 
$\T \mathbb A_{\gg} \to \ggg
$
(component of $\phi_{J}$ into $i\,\ggg $)
has maximal rank.

\noindent
{\rm (ii)}
Every  
$\ggg^{\mathbb C}$-valued $1$-form 
on $\mathbb A_{\gg}$
whose imaginary part has maximal rank
arises in this manner
from an almost complex structure $J$ 
on $ \GG \times \mathbb A_{\gg}$ of the kind spelled out in 
{\rm (i)}.

\noindent
{\rm (iii)}
Under the circumstances of {\rm (i)} or {\rm (ii)},
the  almost complex structure $J$ on $ \GG \times \mathbb A_{\gg}$ is 
integrable if and only
if $\phi_{J}$ satisfies the integrability condition
\begin{equation}
d\phi_{J} + \phi_{J}\wedge^{-}\phi_{J}
\in \mathcal A^2(\mathbb A_{\mathfrak g}, \ggg).
\label{integrability}
\end{equation} 
\end{lem}

\begin{proof}
The pair $(\theta,L)$ satisfies the
integrability conditions \eqref{integra1} and \eqref{integra2}
if and only if
 $\phi_{J}$ 
satisfies the integrability condition
\begin{equation}
d\phi_{J} + \phi_{J}\wedge^{-}\phi_{J}
\in \mathcal A^2(\mathbb A_{\mathfrak g}, \ggg).
\end{equation} 
\end{proof}

For illustration, and to introduce notation, 
suppose that $\GG$ is compact (and connected), and
let
$(\theta_{\mathrm{st}},L_{\mathrm{st}})$ 
be the pair arising from the standard complex structure
$J_{\mathrm{st}}$
on $G^{\mathbb C}$; 
in terms of the obvious trivialization
$\T  \mathbb A_{\gg} \cong  \mathbb A_{\gg} \times \gg$,
the associated
$\ggg^{\mathbb C}$-valued $1$-form 
$\phi_{\mathrm{st}}\colon \T \mathbb A_{\gg} \longrightarrow \ggg^{\mathbb C}$ 
on $\mathbb A_{\gg}$, cf. \eqref{thetaL} and \eqref{associ1} above,
is given by
\begin{equation}
\begin{aligned}
\phi_{\mathrm{st}}
&= (d\gamma_{\mathrm{st}})\gamma_{\mathrm{st}}^{-1}= \ppsi_{\mathrm{st}} +i \llambda_{\mathrm{st}},
\\
(\ppsi_{\mathrm{st}})_\qqq(V)&=\frac{\mathrm{cos}(\mathrm{ad}(\qqq))-\mathrm{Id}}
{\mathrm{ad}(\qqq)} (V)
=\sum_{j=1}^{\infty}\tfrac{(-1)^{j}}{(2j)!}\mathrm{ad}^{2j-1}(\qqq)(V)
\\ 
(\llambda_{\mathrm{st}})_\qqq(V)&= \frac{\mathrm{sin}(\mathrm{ad}(\qqq))}
{\mathrm{ad}(\qqq)} (V)=\sum_{j=0}^{\infty}\tfrac{(-1)^{j}}{(2j+1)!}\mathrm{ad}^{2j}(\qqq)(V),
\end{aligned}
\label{stand1}
\end{equation}
where $\qqq \in \mathbb A_{\gg}$,  $V \in \gg$.
Consequently, given a point $\qqq$ of $\mathbb A_{\gg}$
and, furthermore, $X \in \T_e\GG = \gg =\ggg$ (identifications as real vector spaces)
and $V \in \T_\qqq \mathbb A_{\gg} \cong \gg = \ggg$, the values
$\theta_{\mathrm{st}}(X,V)$ and $L_{\mathrm{st}}(X,V)$
are given by
\begin{align*}
(\theta_{\mathrm{st}})_{(e,\qqq)}
(X,V)&= X+\frac{\mathrm{cos}(\mathrm{ad}(\qqq))-\mathrm{Id}}
{\mathrm{ad}(\qqq)} (V),
\\ 
(L_{\mathrm{st}})_{(e,\qqq)}(X,V)&= \frac{\mathrm{sin}(\mathrm{ad}(\qqq))}
{\mathrm{ad}(\qqq)} (V).
\end{align*}
These are exactly the expressions given as (3.5) and (3.6) in 
\cite{MR1989647}.

\subsection{Proof of Theorem {\rm \ref{fundth24}}}
Let $J$ be an admissible  almost complex structure on 
$ \PGG=\GG \times \mathbb A_{\gg}$.
Let  $\phi_{J}\in \mathcal A^1(\mathbb A_{\mathfrak g},\ggg^{\mathbb C})$ be the associated 
$\ggg^{\mathbb C}$-valued 1-form, cf. Lemma \ref{fundth11},
and suppose that $\phi_{J}$ satisfies the integrability
condition \eqref{integrability}.
Then $\phi_{J}$ integrates to a smooth map
$\gamma_J\colon \mathbb A_{\mathfrak g} \to G^{\mathbb C}$ so that
\begin{equation}
\phi_J=(d\gamma_J)\gamma^{-1}_J.
\label{integra0}
\end{equation}
The map $\gamma_J$ is unique up to a constant in $ G^{\mathbb C}$,
that is,
given $\gamma^1_J$  and $\gamma^2_J$
satisfying \eqref{integra0},
$\gamma^2_J=\gamma^1_Jc$ for some $c \in  G^{\mathbb C}$.
By construction, the associated generalized polar map \eqref{genpolar}
is holomorphic. 
Moreover, the composite
\eqref{compo}
is a smooth map of maximal rank
between smooth manifolds of the same dimension
and hence an immersion. Consequently
$\Pi_{\gamma_J}$ is as well an immersion.

\subsection{Proof of Theorem \ref{fundth241}}

For clarity, we will momentarily 
proceed under somewhat more general circumstances than actually needed for the
proof of the theorem.
Thus, as before, consider
a group $G$, a subgroup $H\subseteq G$, 
and a simply connected (left) $H$-manifold $B$.

\begin{lem} Given a smooth $G$-valued map $\gamma\colon B \to G$,
the associated 
$1$-form ${(d\gamma)\gamma^{-1}\in \mathcal A^1(B,\overline {\mathfrak g})}$
is $H$-equivariant if and only if 
$\gamma$ is quasi $H$-equivariant, that is, if and only if
there is a smooth $G$-valued function
\[
c\colon H \longrightarrow G
\]
on $H$ such that
\[
 \gamma^{-1}(b)\,{}^x\gamma (b)=c(x),\ x \in H,\ b \in B.
\]
\end{lem}

\begin{proof}
This is routine and left to the reader.

\end{proof}

\begin{cor}\label{cor102}
Let
$\bbb \in \mathcal A^1(B,\overline{\mathfrak g})$
be an
$H$-equivariant
 $1$-form that satisfies the integrability condition
$d\bbb+\bbb\wedge^{-}\bbb = 0$, and let
$\gamma\colon B \to G$ integrate $\bbb$ in the sense that
$\bbb= (d\gamma)\gamma^{-1} $.
Then there is a smooth $G$-valued function $c \colon B \to G$
on $B$ such that
\begin{equation}
\gamma(xb) = x \gamma(b)c(x^{-1})x^{-1},\ x \in H,\ b \in B.
\end{equation}
\end{cor}

Now we prove Theorem \ref{fundth241}.
Let $J$ be a biinvariant admissible 
complex structure on $\GG \times \mathbb A_{\gg}$.
View $J$ merely as a left translation invariant complex structure,
let  $(\theta,L)$ denote the associated pair of
$\overline{\gg}$-valued 1-forms,
let
\begin{equation}
\phi_J=
\psi_{\theta} + i \lambda_L\colon  \T \mathbb A_{\gg} \longrightarrow 
\overline{\gg}^{\mathbb C} = \overline{\gg}\oplus i \overline{\gg}
\label{thetaLLL}
\end{equation}
be the resulting integrable 
$\overline{\gg}^{\mathbb C}$-valued 1-form on $\mathbb A_{\gg}$,
cf. \eqref{thetaL}, 
and let
$\gamma_J\colon \mathbb A_{\mathfrak g} \to G^{\mathbb C}$ 
be a smooth map 
that integrates $\phi_J$; thus
$\phi_J$ equals $(d\gamma_J)\gamma^{-1}_J$.
By construction, the associated generalized polar map 
\eqref{genpolar}
is holomorphic, and
the composite
\begin{equation}
\begin{CD}
 \mathbb A_{\mathfrak g}
@>{\gamma_J}>>
 G^{\mathbb C}
@>{\pi}>>
\overline G\big\backslash  G^{\mathbb C}
\end{CD}
\end{equation}
is a smooth map of maximal rank.

Since $J$ is as well right translation invariant,
so is $\phi_J$. By Corollary \ref{cor102},
there is a smooth $\GG^{\mathbb C}$-valued
function $c \colon \GG \to \GG^{\mathbb C}$ such that
\begin{equation}
\gamma_J(zYz^{-1}) = z \gamma_J(Y)c(z^{-1})z^{-1},\ z \in G,\ Y \in \mathbb A_{\gg}.
\end{equation}
Consequently
\begin{equation}
\Pi_{\gamma_J}((x,Y)z)
=
(\Pi_{\gamma_J}(x,Y))c(z) z.
\end{equation}

\subsection{Explicit description of the almost complex structure}
\label{expdesc}

The constructions being left $\GG$-invariant, 
it suffices to 
spell out an explicit expression
for the (almost) complex structure $J$ 
on $\GG \times \mathbb A_{\gg}$ 
in Proposition \ref{fundth1}
in terms of the two 1-forms  
$\ppsi_{\theta}$ and $\llambda_{L}$
 at the points of $\GG \times \mathbb A_{\gg}$  of the kind  $(e,\qqq)$
as $\qqq$ ranges over $\mathbb A_{\gg}$.

\begin{prop}\label{explicit}
Let $\ppsi\colon\T \mathbb A_{\gg} \to \gg$ and
 $\llambda\colon\T \mathbb A_{\gg} \to \gg$
be two 
$\gg$-valued $1$-forms on $\mathbb A_{\gg}$,
the $1$-form $\llambda$ being of maximal rank.
For each $\qqq \in \mathbb A_{\gg}$, the expression
\begin{equation}
J_{(e,\qqq)}(u,v)=
(-\llambda_\qqq(v)-\ppsi_\qqq\llambda^{-1}_\qqq(u+\ppsi_\qqq(v)), \llambda^{-1}_\qqq(u+\ppsi_\qqq(v))),
\label{explicit2}
\end{equation}
as $u$ ranges over $\T_e \GG = \gg$ and $v$ over $\T_\qqq \mathbb A_{\gg} \cong \gg$,
yields a complex structure $J_{(e,\qqq)}$ on the tangent space $\T_{(e,\qqq)} (\GG \times \mathbb A_{\gg})$
to  $\GG \times \mathbb A_{\gg}$ at the  point $(e,\qqq)$ of $\GG \times \mathbb A_{\gg}$, and $\GG$-left translation then yields
an admissible almost complex structure $J$ 
on $\GG \times \mathbb A_{\gg}$.
In particular, when
$(\ppsi,\llambda)$ 
is the pair $(\ppsi_{\theta},\llambda_L)$ of $1$-forms  on $\mathbb A_{\gg}$
arising from a given admissible almost complex structure $J$ 
on $\GG \times \mathbb A_{\gg}$,
the expression {\rm \eqref{explicit2}} yields $J$ in terms of
$\ppsi_{\theta}$ and $\llambda_L$.
\end{prop}

\begin{proof} This is left to the reader.
\end{proof}

\section{Invariant K\"ahler forms}
\label{leftk}

Recall that $\theta_G\colon
\T (\GG \times \mathbb A_{\gg}) \to \ggg$
denotes the trivial principal $\GG$-connection form
(relative to the obvious principal left $\GG$-structure on 
$\GG \times \mathbb A_{\gg}$).
For better readability, we continue to denote $\GG \times \mathbb A_{\gg}$
by $\PGG$ whenever appropriate.
Given a (left) $G$-manifold $M$, we denote the fundamental vector field
on $M$ associated to $X\in \gg$ by $X_M$.

\begin{lem}
Given a $\GG$-equivariant map
$\mu \colon \GG \times \mathbb A_{\gg} \to \gg^*$,
this map $\mu$ is a momentum for ${\oo=-d\langle \mu,\theta_{\GG} \rangle}$
in the sense that
\begin{equation}
i_{X_{\PGG}}\oo= d\langle \mu,X \rangle,\ X \in \gg. 
\label{moma0}
\end{equation}
\end{lem}

\begin{proof}
Let $X \in \gg$ and write $\Theta=\langle \mu,\theta_{\GG} \rangle$. 
Then
\begin{align*}
0=\mathcal L_{X_{\PGG}}(\Theta)&=(di_{X_{\PGG}} +i_{X_{\PGG}}d)(\Theta)
\\
d(\Theta(X_{\PGG}))&=i_{X_{\PGG}}\oo
\\
\Theta(X_{\PGG})&=\langle \mu,\theta_{\GG} \rangle(X_{\PGG})
=\langle \mu,\theta_{\GG}(X_{\PGG}) \rangle
=\langle \mu,X \rangle,
\end{align*}
whence the assertion.
\end{proof}

The following is entirely classical.
\begin{prop}
\label{class1}
Given a hamiltonian $G$-manifold $(M,G,\oo,\mu)$
with $G$-action on $M$ from the left,
\begin{equation}
\oo(X_M,Y_M)
=\langle\mu, [X,Y]\rangle,\ X,Y \in \gg. \qed
\label{whence2}
\end{equation}
\end{prop}

For intelligibility, we will now
recall from Proposition \ref{fundth1} that,
given an admissible (almost) complex structure $J$ on
 $\GG \times \mathbb A_{\gg}$, 
the notation being that established in  Proposition \ref{fundth1},
the $\ggg$-valued $1$-form
\[
\llambda_J^{\GG}=L_J= -\theta_J \circ J\colon \T  \PGG \to\ggg
\]
is basic and regular and that
\begin{align*}
\theta_J&=\theta_G+\ppsi_J^G 
\\
dL_J&=-[\ppsi_J^{\GG},L_J]^{-}=[\ppsi_J^{\GG},L_J] .
\end{align*}

\begin{thm}
\label{kahlert}
Let $J$ be an admissible complex structure on $\GG \times \mathbb A_{\gg}$,
let $\gamma_J\colon \mathbb A_{\gg} \to \GG^{\mathbb C}$
be an admissible map inducing it, and let
$\llambda_J$ and
$\ppsi_J$ be the two associated $\gg$-valued $1$-forms on $\mathbb A_{\gg}$
characterized by the identity
\[
(d \gamma_J)\gamma^{-1}_J=\ppsi_J+i\llambda_J \in \form^1(\mathbb A_{\gg},\gg^{\mathbb C}). 
\]
Moreover let $\mu \colon \GG \times \mathbb A_{\gg} \to \gg^*$
be a $\GG$-equivariant map, and suppose that
$\oo=-d\langle \mu,\theta_{\GG} \rangle$
is symplectic. Then $J$ and $\omega$ combine to a pseudo K\"ahler structure
on $\GG \times \mathbb A_{\gg}$ 
(necessarily having momentum mapping $\mu$)
if and only if the two real $1$-forms
$\langle \mu,\ppsi_J  \rangle$
and 
$\langle \mu,\llambda_J \rangle$
on $\mathbb A_{\gg}$ are closed.
Furthermore, the K\"ahler form $\oo$  
is then given by the expression
\begin{align}
\oo &=-d\langle \mu,\theta_J \rangle. 
\label{kaehlerform}
\end{align}
Finally, when 
$f\colon \GG \times \mathbb A_{\gg} \to \mathbb R$ is the $G$-invariant extension of an 
integral of $\langle \mu,\llambda_J \rangle$  on $\mathbb A_{\gg}$,
i.~e., when $f$ is a
smooth real $G$-invariant function 
on $\GG \times \mathbb A_{\gg}$ whose restriction to
$\{e\}\times \mathbb A_{\gg}$ satisfying the identity
\begin{equation}
df=\langle \mu,\llambda_J \rangle,
\label{integ}
\end{equation}
the function $2f$ is a K\"ahler potential
on $\GG \times \mathbb A_{\gg}$, that is, satisfies the identity
\begin{equation}
\oo=2i\partial \overline \partial f.
\label{kaehlerpo}
\end{equation}
\end{thm}

\begin{lem}
\label{circums}
Under the hypotheses of Theorem {\rm \ref{kahlert}},
save that we de not assume that $\oo=-d\langle \mu, \theta_G\rangle$
is non-degenerate,
the following hold.

\noindent
{\rm (i)} 
The data  $\theta_J$, $L_J$ and $\mu$ determine a smooth 
$\GG$-equivariant map
\begin{equation}
\Psi\colon\GG \times \mathbb A_{\gg} \to  
\mathrm{Hom}(\gg, \gg^*)
\label{Psi}
\end{equation}
satisfying the identity
\begin{equation}
\label{phii}
 d^\theta \mu = \Psi \circ L_J\colon \T (\GG \times \mathbb A_{\gg}) 
\longrightarrow\gg
\longrightarrow  \gg^*
\end{equation}
in the sense  that, for any $(x,\qqq)\in \GG \times \mathbb A_{\gg}$, 
the composite
\begin{equation*}
\begin{CD}
 \T_{(x,\qqq)} (\GG \times \mathbb A_{\gg})
@>{(L_J)_{(x,\qqq)}}>>  \gg
@>{\Psi_{(x,\qqq)}}>>
\gg^*
\end{CD}
\end{equation*}
coincides with $(d^\theta\mu)_{(x,\qqq)}$. 

\noindent
{\rm (ii)} 
Given $X,Y \in \gg$,
\begin{equation}
\oo (X_{\PGG},JY_{\PGG})
=
\langle X,\Psi(Y)\rangle .
\end{equation}

\noindent
{\rm (iii)} The map 
$\Psi$ is symmetric in the sense that
$\langle X,\Psi(Y)\rangle= \langle Y,\Psi(X)\rangle$
for any $X,Y \in \gg$ if and only if the real $1$-form
$\langle \mu,\llambda_J  \rangle$
on $\mathbb A_{\gg}$ is closed.
\end{lem}

\noindent{\bf Complement to Theorem \ref{kahlert}}
{\em Under the circumstances of the theorem,
the metric $\rg$  
is given by the expression
\begin{align}
\rg&= \langle \Psi \theta,\theta\rangle +\langle L, d^\theta \mu \rangle 
+\langle \mu,[L,\theta]\rangle.
\label{rmetric}
\end{align}
}
\begin{proof}[Proof of Lemma {\rm \ref{circums}}]
Since the $\gg$-valued 1-form $\llambda_J$ has maximal rank everywhere,
the data  $\theta_J$, $L_J$ and $\mu$ determine a smooth 
map  
$\Psi\colon\mathbb A_{\gg} \to  
\mathrm{Hom}(\gg, \gg^*)$
 satisfying
the identity
\begin{equation}
\label{phiii}
 d^\theta \mu = \Psi \circ \llambda_L\colon \T \mathbb A_{\gg} 
\longrightarrow\gg
\longrightarrow  \gg^*
\end{equation}
in the sense  that, for every point $\qqq$ of
$\mathbb A_{\gg}$, the composite
\begin{equation*}
\begin{CD}
 \T_\qqq \mathbb A_{\gg}\cong \gg @>{(\llambda_J)_\qqq}>>  \gg
@>{\Psi_\qqq}>>
\gg^*
\end{CD}
\end{equation*}
coincides with $(d^\theta\mu)_\qqq$. 
The map \eqref{Psi} is then the unique $\GG$-equivariant extension
to all of $\GG \times \mathbb A_{\gg}$.
This establishes (i).

Let $X,Y\in \gg$. By construction,
\begin{align*}
\Psi(X)=\Psi(L_J(JX_{\PGG}))&=
\Psi(\theta_J(X_{\PGG}))=
 d^\theta \mu (JX_{\PGG})
\\
\oo (X_{\PGG},JY_{\PGG})
&=
\langle X,d\mu(JY_{\PGG})\rangle
\\
&=
\langle X,d^{\theta}\mu(JY_{\PGG})\rangle\ \textrm{since}\ JY_{\PGG}\ 
\textrm{horizontal}
\\
&=
\langle X,\Psi(L_J(JY_{\PGG}))\rangle
\\
&=
\langle X,\Psi(Y)\rangle,
\end{align*}
whence (ii) holds.

Finally to prove (iii), we note first that the integrability
condition \eqref{integra2}, viz. $d^{\theta_J}L_J=0$, implies that
\begin{equation}
d\langle \mu, L_J \rangle=\langle d^{\theta_J}\mu\wedge L_J \rangle.
\end{equation}
Let $X,Y \in \gg$.
Then
\begin{align*}
\langle d^{\theta}\mu\wedge L_J \rangle
(JX_{\PGG},JY_{\PGG})
&=
\langle d^{\theta}\mu(JX_{\PGG}),
L_J(JY_{\PGG})) \rangle
-
\langle d^{\theta}\mu(JY_{\PGG}),
L_J(JX_{\PGG})) \rangle
\\
&=
\langle \Psi(X),Y \rangle
-
\langle \Psi(Y),X \rangle.
\end{align*}
Hence the closedness of $\langle \mu, L_J \rangle$
implies the symmetry of $\Psi$. Since the 1-form $L_J$ is
basic, it vanishes on vertical vectors whence the symmetry
of  $\Psi$ implies
the closedness of $\langle \mu, L_J \rangle$.
\end{proof}

\begin{rema}
Under the hypotheses of Theorem {\rm \ref{kahlert}},
when $J$ and $\oo$ combine to a pseudo K\"ahler structure,
in view of Lemma \ref{circums} (ii),
for any point $(x,\qqq)$ of $\GG \times \mathbb A_{\gg}$, the linear map
$\Psi_{(x,\qqq)}\colon \gg \to \gg^*$ is invertible, and
the constituent
$\langle L, d^\theta \mu \rangle$ of \eqref{rmetric}
can be written as
\begin{equation}
\langle L, d^\theta \mu \rangle 
=
\langle \Psi^{-1}d^\theta \mu, d^\theta \mu \rangle .
\label{constitu}
\end{equation}
In particular, $\langle L, d^\theta \mu \rangle$ is a symmetric bilinear
form, and the metric \eqref{rmetric} can be written as
\begin{equation}
\rg= \langle \Psi \theta,\theta\rangle +
\langle \Psi^{-1}d^\theta \mu, d^\theta \mu \rangle
+\langle \mu,[L,\theta]\rangle .
\label{rmetric2}
\end{equation}
This is essentially the same expression as 
\cite{MR1989647} (4.2).
\end{rema}

\begin{rema}
\label{fundv}
At every point of $\PGG=\GG \times \mathbb A_{\gg}$,
the tangent space is spanned by vectors arising from
fundamental vector fields $X_{\PGG}$ and vector fields
of the kind $JY_{\PGG}$, as $X$ and $Y$ range over $\gg$.
\end{rema}

\begin{lem}
\label{circums2}
Under the hypotheses of Theorem {\rm \ref{kahlert}},
when $J$ and $\oo$ combine to a pseudo K\"ahler structure, 
the
metric $\rg=\oo(\,\cdot\, , J \,\cdot\,)$
is given by the expression {\rm \eqref{rmetric}}
and the K\"ahler form $\oo$ by  the expression {\rm \eqref{kaehlerform}}.
\end{lem}

\begin{proof} Let $X,Y \in \gg$. A straightforward calculation yields
\begin{align*}
\langle \Psi \theta_J,\theta_J\rangle(X_{\PGG},Y_{\PGG})&=
\rg(X_{\PGG},Y_{\PGG})
\\
\langle \mu,[L_J,\theta_J]\rangle(JX_{\PGG},Y_{\PGG})&= 
\rg(JX_{\PGG},Y_{\PGG})
\\
\langle \mu,[L_J,\theta_J]\rangle(X_{\PGG},JY_{\PGG})
&=\rg(X_{\PGG},JY_{\PGG})
\\
\langle L_J, d^{\theta_J} \mu \rangle(JX_{\PGG},JY_{\PGG})&=
\rg(JX_{\PGG},JY_{\PGG}),
\end{align*}
and evaluation of any of the three constituents on the right-hand side of \eqref{rmetric} at argument pairs not  already spelled out is zero.
In view of Remark \ref{fundv},
this shows that
the
metric $\rg=\oo(\,\cdot\, , J \,\cdot\,)$
is given by the expression {\rm \eqref{rmetric}}.

Likewise, in view of \eqref{integra1}, 
we find
\begin{align*}
\oo &= \rg(J\,\cdot\, ,\,\cdot\,)
=\langle \Psi \theta J,\theta\rangle +\langle L J, d^\theta \mu \rangle 
+\langle \mu,[L,\theta](J\,\cdot\, ,\,\cdot\,)\rangle
\\
\langle \Psi \theta J,\theta\rangle
& 
=-\langle d^{\theta}\mu,\theta\rangle
\\
\langle L J, d^\theta \mu \rangle &=
\langle  \theta,  d^{\theta}\mu \rangle
\\
\langle \mu,[L,\theta](J\,\cdot\, ,\,\cdot\,)\rangle
&=\langle \mu, \theta\wedge \theta -L\wedge L \rangle
=-\langle \mu, d\theta \rangle
\\
\rg(J\,\cdot\, ,\,\cdot\,)
&=-(\langle d^{\theta}\mu,\theta\rangle
   -\langle  \theta,  d^{\theta}\mu \rangle
  +\langle \mu, d\theta \rangle)
\\
&=-(\langle d\mu,\theta\rangle
   -\langle  \theta,  d\mu \rangle
  +\langle \mu, d\theta \rangle)
\\
&\quad
-(\langle [c_J^{\GG},\mu]^{-},\theta\rangle
   -\langle  \theta,   [c_J^{\GG},\mu]^{-} \rangle)
\\
&=-d\langle \mu, \theta \rangle.
\end{align*}
Consequently the K\"ahler form $\oo$ is given 
by  the expression {\rm \eqref{kaehlerform}}.
\end{proof}

\begin{proof}[Proof of Theorem {\rm \ref{kahlert}}]
The \lq\lq Furthermore statement\rq\rq\ has been established already in Lemma \ref{circums2}.

Suppose first that $\langle \mu,c_J^G \rangle$ and
$\langle \mu,L_J \rangle$ 
are closed. We must show that $J$ is compatible with $\oo$.
Since $\langle \mu,\ppsi_J^G\rangle$ is closed,
\begin{align*}
d\langle \mu,\theta_J \rangle&=d\langle \mu,\theta_G \rangle
+d\langle \mu,  \ppsi_J^G\rangle
\\
&=
d\langle \mu,\theta_G \rangle =-\oo,
\end{align*}
that is, $\langle \mu,\theta_J\rangle$ is a symplectic potential for $\oo$.

Let $X,Y \in \gg$.
Since $JX_{\PGG}$ and $JY_{\PGG}$ are horizontal, 
since
$\ppsi_J^G$ vanishes on vertical vector fields,
and since $[ X_{\PGG},Y_{\PGG}]= -[X,Y]_{\PGG}$,
we find
\begin{align*}
\oo(JX_{\PGG},JY_{\PGG})&=d\langle \mu,\theta_J \rangle(JY_{\PGG},JX_{\PGG})
\\
&=\langle d\mu(JY_{\PGG}),\theta_J(JX_{\PGG})\rangle 
-
\langle d\mu(J X_{\PGG}),\theta_J(JY_{\PGG})\rangle
\\
&
\quad
+\langle \mu,\theta_J \rangle[J X_{\PGG},JY_{\PGG}]
\\
&=
\langle \mu,\theta_J \rangle[J X_{\PGG},JY_{\PGG}]
\\
&=
-\langle \mu,\theta_J [ X_{\PGG},Y_{\PGG}]\rangle
\\
&=
\langle \mu,\theta_J [X,Y]_{\PGG}\rangle
\\
&=
\langle \mu, [X,Y]\rangle.
\end{align*}
In view of Proposition \ref{class1}, we deduce
\begin{equation}
\oo(JX_{\PGG},JY_{\PGG})=\oo(X_{\PGG},Y_{\PGG}).
\label{hence}
\end{equation}

Next, by Lemma \ref{circums} (ii),
\begin{align*}
\oo (X_{\PGG},JY_{\PGG})
&=
\langle X,\Psi(Y)\rangle 
\\
\oo (Y_{\PGG},JX_{\PGG})
&=
\langle Y,\Psi(X)\rangle, 
\end{align*}
and,
since 
$\oo (JX_{\PGG},JJY_{\PGG})
=
\oo (Y_{\PGG},JX_{\PGG})$,
by Lemma \ref{circums} (iii),
\[
\oo (X_{\PGG},JY_{\PGG})
=
\oo (JX_{\PGG},JJY_{\PGG}),
\]
since the real 1-form $\langle \mu, L_J\rangle$
is closed.
In view of Remark \ref{fundv},
these calculations show that the closedness of
$\langle \mu,\ppsi_J^G \rangle$ and
$\langle \mu,L_J \rangle$ 
implies that $J$ is compatible with $\oo$.

Conversely, suppose that $J$ is compatible with $\oo$.
Then
\[
\oo (X_{\PGG},JY_{\PGG})
=
\oo (JX_{\PGG},JJY_{\PGG})
\]
and,
since 
$\oo (JX_{\PGG},JJY_{\PGG})
=
\oo (Y_{\PGG},JX_{\PGG})$,
we conclude
\begin{equation*}
\langle X,\Psi(Y)\rangle =\oo (X_{\PGG},JY_{\PGG})
=
\oo (Y_{\PGG},JX_{\PGG})
=
\langle Y,\Psi(X)\rangle, 
\end{equation*}
whence, by Lemma \ref{circums} (iii), the
real 1-form $\langle \mu, L_J\rangle$
is closed.
Finally, by Lemma \ref{circums2},
\[
\oo = -d\langle \mu,\theta_J\rangle = -d\langle \mu,\theta_{\GG}\rangle
-d\langle \mu,\ppsi_J^{\GG}\rangle.
\]
Since $\oo = -d\langle \mu,\theta_{\GG}\rangle$,
we conclude that the real 1-form $\langle \mu,\ppsi_J^{\GG}\rangle$
is closed.

To establish the \lq\lq Finally\rq\rq\ assertion we recall that,
by construction, cf. \eqref{L},
$
L=-\theta\circ J$.
Using the fact that, relative to the decomposition
$\T M\otimes \mathbb C = 
\T^{\mathrm{hol}} M \oplus  \overline{\T^{\mathrm{hol}} M}$
of the total space $\T M\otimes \mathbb C$ of the complexified tangent bundle
of $M= \GG \times \mathbb A_{\gg}$
into its holomorphic and antiholomorphic constituents,
keeping in mind that, on 
$\T^{\mathrm{hol}} M$,
the complex structure is given by multiplication by $i$
and on $\overline{\T^{\mathrm{hol}} M}$
 by multiplication by $-i$, we find
\begin{align*}
\omega & =  - d \langle \mu, \theta \rangle 
 =  - d ( \langle \mu, L\circ J \rangle ) 
 =  - d (\langle \mu, L \rangle \circ J) 
\\
& =  - d ((df) \circ J) 
 =   - d ((\partial f) \circ J + 
(\overline \partial f) \circ J) 
\\
&
= - (\partial + \overline \partial) (i(\partial f)  - 
i(\overline \partial f)) 
=- i\overline \partial \partial f
+ i \partial \overline \partial f 
\\
&
 =  2i \partial \overline \partial f. 
\end{align*}
\end{proof}

\section{The case when $\GG$ is compact}
\label{standard}
Suppose that $\GG$ is compact and connected.
Pick an invariant inner product 
$\,\cdot\,\colon \gg \times \gg \to \mathbb R$ on $\gg$ and use it to identify
$\gg$ with $\gg^*$. 
The induced biinvariant Riemannian metric on $\GG$
identifies $\T \GG$ with $\T^* \GG$ in a $\GG$-biequivariant manner.
Let $\mu\colon \GG \times \mathbb A_{\gg} \to \gg^*$
be the $\GG$-equivariant map 
 ${\GG \times \mathbb A_{\gg} \to \gg^*}$
which, restricted to $\{e\}\times \mathbb A_\gg$,
is the adjointness isomorphism
$\sharp\colon \gg \to \gg^*$
 of the inner product on $\gg$
(the identity of $\gg$ when we identity
$\gg^*$ with $\gg$ via $\sharp$).
Under the resulting $\GG$-biequivariant
identification
\[
\T^* \GG \longrightarrow \T \GG \longrightarrow \GG \times \mathfrak
{\mathbb A}_\gg,
\]
the  $\GG$-biinvariant $1$-form
$\langle \mu, \theta_G\rangle$
on $\GG \times \mathbb A_{\gg}$ corresponds to the
tautological $1$-form on $\T^* \GG$, and $\mu$ is the 
momentum mapping for 
$\oo=-d\langle \mu, \theta_G\rangle$,
uniquely determined by $\oo$ up to a central value.
By construction, under the identification of
$\T^*\GG$ with $\T\GG$ induced by the chosen inner product on $\gg$, the standard cotangent bundle symplectic structure
corresponds to $\oo$.
\subsection{K\"ahler structures on $\T^* \GG$}

Let $J$ be the standard complex
structure on $\GG \times \mathbb A_{\gg}$, cf. Subsection \ref{ordinarypo}.
In view of \eqref{stand1},
the  $\GG$-invariant $1$-form
$\langle \mu,  \ppsi^G_{\mathrm{st}} \rangle$
is zero,
and 
\begin{equation}
\langle \mu,  \llambda^G_{\mathrm{st}} \rangle_{\qqq}(V)
=\sum_{j=0}^{\infty}\tfrac{(-1)^{j}}{(2j+1)!}\qqq \cdot (\mathrm{ad}^{2j}(a)(V)) =\qqq \cdot V,
\label{stand2}
\end{equation}
where $\qqq \in \mathbb A_{\gg}$,  $V \in \gg$.
Thus 
$\langle \mu,  \llambda^G_{\mathrm{st}} \rangle$
is plainly closed, indeed 
\begin{equation}
\langle \mu,  \llambda^G_{\mathrm{st}} \rangle =df,\ 
f(\qqq)= \tfrac 12 \qqq \cdot \qqq,\ \qqq \in \mathbb A_{\gg}.
\label{funct1}
\end{equation}
Theorem \ref{kahlert}
entails that $J$ and $\oo$ combine to
a pseudo K\"ahler structure,
and this structure is actually positive, that is, an 
ordinary K\"ahler structure.
This structure coincides with the 
 K\"ahler structure induced from the standard K\"ahler structure
on  $\GG^{\mathbb C}$
(relative to the chosen inner product on $\gg$)
via the
ordinary polar map
 $\GG \times \mathfrak
{\mathbb A}_\gg
\longrightarrow \GG^{\mathbb C}$, and
we will refer to this structure 
on $\GG \times \mathfrak
{\mathbb A}_\gg$
as the {\em standard K\"ahler structure
relative to the chosen inner product \/} $\cdot\,$ on $\gg$;
cf. Example 4.3 in \cite{MR1989647}, \cite {MR1892462}.
We denote
the $\GG$-biinvariant 
extension
of the function $f$ introduced in \eqref{funct1} above 
to all of $\GG \times \mathbb A_{\gg}$
still by $f$; 
in view of Theorem \ref{kahlert}, 
the function
$2f$ is actually a K\"ahler potential.
We note that, given the point $\qqq$ of $\mathbb A_{\gg}$,
\begin{eqnarray}
(d^{\theta}\mu)_{(e,\qqq)}
=&
\sharp \circ \mathrm{cos}(\mathrm{ad}(\qqq))&\colon \gg \to \gg^*
\\
\Psi^{-1}_{(e,\qqq)}
=&
\sharp \circ \mathrm{ad}(\qqq)\cot (\mathrm{ad} (\qqq))&\colon \gg \to \gg^* ,
\end{eqnarray}
cf. \eqref{Psi}
for the definition of the linear isomorphism
$\Psi_{(e,\qqq)}\colon \gg^* \to \gg$.
In particular, 
let $\mathbb A_{\ttt}\subseteq \mathbb A_{\gg}$
denote the Lie algebra of a chosen maximal torus $T$ in $\GG$,
viewed as an affine subspace. Consider the associated positive  real roots
$\alpha_1,\ldots,\alpha_m$ characterized by
the convention that, on the root space 
$\gg_{\alpha_j}$
associated to $\alpha_j$,
when $Z_j\in \ttt $ denotes a root vector and $X\in \gg_{\alpha_j}$,
the value $[Z_j,X]$ is given by $i\alpha_j(Z_j) X$.
Now,
when $\qqq$ lies in $\mathbb A_{\ttt}$, 
the inverse $\Psi^{-1}_{(e,\qqq)}$
is given by the expression
\begin{equation}
\Psi^{-1}_{(e,\qqq)} =\mathrm{diag}(1,\ldots,1,
\alpha_1(\qqq)\coth(\alpha_1(\qqq)),\ldots, 
\alpha_m(\qqq)\coth(\alpha_m(\qqq))) 
\end{equation}
and is plainly symmetric; here the number of 1's is equal to the rank of
$\gg$.

More generally, the following hold.

\begin{thm}
 \label{kahlertco}
On 
$\T ^*\GG$, identified with $\GG \times \mathbb A_{\gg}$
by means of the chosen invariant inner product on $\gg$
and by means of left translation,
let $J$ be an admissible complex structure,
let $\gamma_J\colon \mathbb A_{\gg} \to \GG^{\mathbb C}$
be an admissible map inducing $J$, and let
$\llambda_J$ and
$\ppsi_J$ be the two associated $\gg$-valued $1$-forms on $\mathbb A_{\gg}$
characterized by the identity
\[
(d \gamma_J)\gamma^{-1}_J=\ppsi_J+i \llambda_J\in \form^1(\mathbb A_{\gg},\gg^{\mathbb C}).
\]
Furthermore, as before,
let $\mu$ denote the $\GG$-equivariant map
$\GG \times \mathbb A_{\gg} \to \gg^*$ which,
restricted to $\{e\}\times \mathbb A_{\gg}$, is the adjointness
isomorphism $\gg \to \gg^*$ of the chosen inner product on $\gg$, and
let $\oo=-d\langle \mu,\theta_{\GG}\rangle$, so that
$\mu$ is a momentum mapping for $\oo$.
Finally, take
the complexification $\GG^{\mathbb C}$
to be endowed with the standard K\"ahler structure relative to the
chosen inner product on $\gg$.
The following are equivalent.

\noindent
{\rm (i)} The pieces of structure
 $J$ and $\omega$ combine to a K\"ahler structure
on $\GG \times \mathbb A_{\gg}$.

\noindent
{\rm (ii)} 
The two real $1$-forms
$\langle \mu,\ppsi_J  \rangle$
and $\langle \mu,\llambda_J \rangle$
on $\mathbb A_{\gg}$ are closed. 

\noindent
{\rm (iii)} 
The associated generalized polar map
$\Pi_{\gamma_J}\colon \GG \times \mathbb A_{\gg} \longrightarrow 
\GG^{\mathbb C}$ made explicit above as {\rm \eqref{genpolar}}
preserves the symplectic (and hence K\"ahler) structures. 

\noindent
The metric $\rg$
on $\GG \times \mathbb A_{\gg}$
is then given by the expression {\rm \eqref{rmetric}}, and an integral
of   $\langle \mu,\llambda_J \rangle$ yields a K\"ahler potential.
\end{thm}

Indeed, Theorem \ref{kahlert}
implies at once that
(i) and (ii) are equivalent,
with 
\lq\lq pseudo K\"ahler structure\rq\rq\ 
substituted for \lq\lq K\"ahler structure\rq\rq\ in (i).
By construction, the complex structure $J$ on 
$\GG \times \mathbb A_{\gg}$ is induced
from the complex structure on 
$\GG^{\mathbb C}$
via the  generalized polar map
$\Pi_{\gamma_J}$. 
The closedness of the 1-form $\langle \mu,\ppsi_J \rangle$
on $\mathbb A_{\gg}$ is equivalent to 
$\Pi_{\gamma_J}$ being compatible with the symplectic structures.
We justify this claim in the next subsection.

\begin{rema}\label{positivity}
While Theorem \ref{kahlert}
entails that, in the statement of Theorem \ref{kahlertco},
(i) and (ii) are equivalent,
with 
\lq\lq pseudo K\"ahler structure\rq\rq\ 
substituted for \lq\lq K\"ahler structure\rq\rq\ in (i),
under the circumstances of  Theorem \ref{kahlertco},
the positivity of the resulting pseudo K\"ahler structure
on  $\GG \times \mathbb A_{\gg}$ is automatic since this structure is induced
from the standard K\"ahler structure 
 on 
$\GG^{\mathbb C}$ via the associated generalized polar map.
\end{rema}

\subsection{Factorization of the generalized polar map}
\label{factor}

Let $\chi=(\chi_{\GG},\chi_{\gg})\colon 
\mathbb A_{\gg}\to G\times \mathbb A_{\gg}$
be a smooth map whose component
$\chi_{\gg}\colon \mathbb A_{\gg}\to \mathbb A_{\gg}$
is a local diffeomorphism, necessarily onto an open subset of
$\mathbb A_{\gg}$, and let $\gamma_{\chi}\colon 
\mathbb A_{\gg} \longrightarrow 
\GG^{\mathbb C}$
be the composite of $\chi$ with 
the ordinary polar map $\Pi_{\mathrm{st}}$, that is, 
$\gamma_{\chi}$ is given by the expression
\begin{equation}
\gamma_{\chi}(\qqq)=\chi_{\GG}(\qqq)\gamma_{\mathrm{st}}(\chi_{\gg}(\qqq))=
\chi_{\GG}(\qqq)\mathrm{exp}(i\chi_{\gg}(\qqq)),\ \qqq \in \mathbb A_{\gg}.
\label{expr}
\end{equation}
In terms of the map
\begin{equation}
\Pi_{\chi}\colon \GG\times \mathbb A_{\gg}
\longrightarrow
 \GG\times \mathbb A_{\gg},\ 
\Pi_{\chi}(x,\qqq)=(x\chi_{\GG}(\qqq),\chi_{\gg}(\qqq)),
\label{Pichi}
\end{equation}
the generalized polar map $\Pi_{\gamma_{\chi}}$ associated to $\gamma_{\chi}$
 factors as
\begin{equation}
\begin{CD}
 \GG\times \mathbb A_{\gg}
@>{\Pi_{\chi}}>>
 \GG\times \mathbb A_{\gg}
@>{\Pi_{\mathrm{st}}}>>
\GG^{\mathbb C} .
\end{CD}
\end{equation}
Via the maps
\begin{align*}
\Pi_{\chi_{\gg}}&\colon \GG\times \mathbb A_{\gg}
\longrightarrow
 \GG\times \mathbb A_{\gg},\ 
\Pi_{\chi_{\gg}}(x,\qqq)=(x,\chi_{\gg}(\qqq)),
\\
\Pi_{\chi_{\GG}}&\colon \GG\times \mathbb A_{\gg}
\longrightarrow
 \GG\times \mathbb A_{\gg},\ 
\Pi_{\chi_{\GG}}(x,\qqq)=(x\chi_{\GG}(\qqq),\qqq),
\end{align*}
\eqref{Pichi} plainly decomposes as
\begin{equation}
\begin{CD}
\GG\times \mathbb A_{\mathfrak g}
@>{\Pi_{\chi_{\GG}}}>>
 \GG\times \mathbb A_{\mathfrak g}
@>{\Pi_{\chi_{\gg}}}>>
 \GG\times \mathbb A_{\mathfrak g},
\end{CD}
\label{prinPichi2}
\end{equation}
and $\Pi_{\chi_{\GG}}$ is simply a gauge transformation.
Since $\gg$ is supposed to be compact, every smooth $\gamma_J\colon
\mathbb A_{\gg}\to \GG\times \mathbb A_{\gg}
$ of the kind in 
Theorem \ref{fundth24} has the form $\gamma_{\chi}$ for some
$\chi$ since the ordinary polar  map
$\Pi_{\mathrm{st}}$, cf. \eqref{polardec},
is a diffeomorphism.

\begin{prop}\label{chiexp}
Let $J=J_{\gamma_{\chi}}$ be the induced admissible
complex structure on $\GG \times \mathbb A_{\gg}$, and let
$(\theta,L)$ denote the pair of $\ggg$-valued $1$-forms
on $\GG \times \mathbb A_{\gg}$
associated to $J$ by the construction in Proposition {\rm\ref{fundth1}}.
Then 
\begin{align}
\theta&=
\Pi_{\chi}^* \theta_{\mathrm{st}}\colon \T (G\times \mathbb A_{\gg}) 
\longrightarrow \ggg 
\label{id1}
\\
L&=
\Pi_{\chi}^* L_{\mathrm{st}}\colon \T (G\times \mathbb A_{\gg}) 
\longrightarrow \ggg . 
\label{id2}
\end{align}
Moreover, the constituents $\ppsi_{\chi}$ and $\llambda_{\chi}$ of the resulting
$\ggg^{\mathbb C}$-valued $1$-form
\begin{equation*}
\phi_{(\theta,L)}
= (d\gamma_{\chi})\gamma_{\chi}^{-1}= \ppsi_{\chi} +i \llambda_{\chi}
\end{equation*}
on $\mathbb A_{\gg}$ given as {\rm \eqref{thetaL}} above take the form
\begin{align}
\ppsi_{\chi}&=(d\chi_{\GG}) \chi_{\GG}^{-1} +\mathrm{Ad}_{\chi_{\GG}}
(\ppsi_{\mathrm{st}}\circ d\chi_{\gg})
\label{id3}
\\ 
\llambda_{\chi}&=\mathrm{Ad}_{\chi_{\GG}}(\llambda_{\mathrm{st}}\circ d\chi_{\gg}).
\label{id4}
\end{align}
\end{prop}

Indeed, \eqref{Pichi}
is a morphism 
of trivial principal left $\GG$-bundles,
spelled out on the total spaces.
The 
naturality of the constructions implies the identities
\eqref{id1} and \eqref{id2}.
In view of \eqref{expr}, 
with the standard notation
$d\chi_{\gg}\colon \T \mathbb A_{\gg} \to  \T \mathbb A_{\gg}$
for the derivative of the map
$\chi_{\gg}\colon \mathbb A_{\gg} \to  \mathbb A_{\gg}$,
the identities \eqref{id1} and \eqref{id2} imply
the identities \eqref{id3} and \eqref{id4}.

\begin{proof}[Proof of Theorem {\rm\ref{kahlertco}}] 
By construction
\begin{align*}
\theta_{\mathrm{st}}&=
\theta_G+\ppsi^G_{\mathrm{st}}
\\
\theta&=\Pi_{\chi}^* \theta_{\mathrm{st}}=
\theta_G+\ppsi^G_{\chi}
\\
\ppsi_{\chi}&=\ppsi_J=(d\chi_{\GG}) \chi_{\GG}^{-1} +\mathrm{Ad}_{\chi_{\GG}}
(\ppsi_{\mathrm{st}}\circ d\chi_{\gg})
\\
\Pi_{\chi}^* \langle \mu,\theta_{\mathrm{st}}\rangle&=\langle \mu,\Pi_{\chi}^* \theta_{\mathrm{st}}\rangle
=\langle \mu,\theta_G+\ppsi^G_{\chi}\rangle
\\
\Pi_{\chi}^*\oo &=-\Pi_{\chi}^* d\langle \mu,\theta_{\mathrm{st}}\rangle
=-d\Pi_{\chi}^* \langle \mu,\theta_{\mathrm{st}}\rangle
=-d\langle \mu,\theta_G+\ppsi^G_{\chi}\rangle
\\
&=
\oo -d\langle \mu,\ppsi^G_{\chi}\rangle=
\oo -d\langle \mu,\ppsi^G_{J}\rangle. 
\end{align*}
Hence $\Pi_{\chi}^*\oo=\oo$ if and only if
$\langle \mu,\ppsi^G_{\chi}\rangle$ is closed.
Consequently
the closedness of the 1-form 
$\langle \mu,\ppsi_J \rangle$
on $\mathbb A_{\gg}$
is equivalent to 
$\Pi_{\gamma_J}$ being compatible with the symplectic structures.
\end{proof}

\begin{cor}
\label{cor1}
Given $\chi=(\chi_{\GG},\chi_{\gg})\colon 
\mathbb A_{\gg}\to G\times \mathbb A_{\gg}$
such that the component
$\chi_{\gg}\colon \mathbb A_{\gg}\to \mathbb A_{\gg}$
is a local diffeomorphism,
the induced admissible
complex structure $J=J_{\gamma_{\chi}}$ on $\GG \times \mathbb A_{\gg}$
combines with the symplectic structure
$\oo=-d\langle \mu, \theta_G\rangle$
(the standard structure 
relative to the chosen invariant inner product on $\gg$)
to a K\"ahler structure
on  $\GG \times \mathbb A_{\gg}$ if and only if the
 $1$-forms $ \langle \mu,\ppsi_{\chi}\rangle$ and
$\langle \mu,\llambda_{\chi}\rangle$ 
on  $\mathbb A_{\gg}$ are closed.
\end{cor}

\subsection{Non-standard examples of biinvariant K\"ahler structures}
\label{nontrivial}
For the special case where $\chi_{\GG}$ has constant value $e$,
the conditions in Corollary \ref{cor1} take the form
$
d\langle \mu,\ppsi_{\mathrm{st}}\circ d\chi_{\gg}\rangle
=0$
and
$
d\langle \mu,\llambda_{\mathrm{st}}\circ d\chi_{\gg}\rangle
=0$.
The $\gg$-valued 1-forms $\ppsi_{\mathrm{st}}\circ d\chi_{\gg}$ 
and
$\llambda_{\mathrm{st}}\circ d\chi_{\gg}$
have the form
\[
\begin{CD}
\T \mathbb A_{\gg}
@>{d\chi_{\gg}}>>
\T \mathbb A_{\gg}
@>{\ppsi_{\mathrm{st}}}>> \gg
\\
\T \mathbb A_{\gg}
@>{d\chi_{\gg}}>>
\T \mathbb A_{\gg}
@>{\llambda_{\mathrm{st}}}>> \gg.
\end{CD}
\]
Let $\qqq \in \mathbb A_{\gg}$ and
$V\in \T \mathbb A_{\gg}\cong \gg$;
by construction,
\begin{align*}
(\ppsi_{\mathrm{st}})_\qqq(V)&=\frac{\mathrm{cos}(\mathrm{ad}(\qqq))-\mathrm{Id}}
{\mathrm{ad}(\qqq)} (V)
\\
(\llambda_{\mathrm{st}})_\qqq(V)&=\frac{\mathrm{sin}(\mathrm{ad}(\qqq))}
{\mathrm{ad}(\qqq)} (V).
\end{align*}
Hence
\begin{align*}
\langle \mu,\ppsi_{\chi}\rangle_\qqq(V)
&= \qqq\cdot((\ppsi_{\chi})_\qqq(V))
\\
&
=
-\tfrac 12
\qqq \cdot
([\chi_{\gg}(\qqq),(d\chi_{\gg})_\qqq(V)])
\\
&
\quad
+\tfrac 1{4!}
\qqq\cdot 
([\chi_{\gg}(\qqq), [\chi_{\gg}(\qqq),[\chi_{\gg}(\qqq), (d\chi_{\gg})_\qqq(V)]]])
\\
&
\quad \pm  \ldots \ .
\\
\langle \mu,\llambda_{\chi}\rangle_\qqq(V)
&= \qqq\cdot((\llambda_{\chi})_\qqq(V))
\\
&
=a\cdot ((d\chi_{\gg})_\qqq(V))-\tfrac 1{3!}
\qqq \cdot
([\chi_{\gg}(\qqq),[\chi_{\gg}(\qqq),(d\chi_{\gg})_\qqq(V)]]) \pm \ldots
\end{align*}
For example, when $[\chi_{\gg}(\qqq),\qqq]$
is zero for every $\qqq \in \gg$,
the 1-form $\langle \mu,\ppsi_{\chi}\rangle$ is even zero
rather than just closed,
and
\begin{align*}
\langle \mu,\llambda_{\chi}\rangle_\qqq(V)
&=
a\cdot ((d\chi_{\gg})_\qqq(V)).
\end{align*}

To construct non-standard examples,
the idea is now to rescale the identity of $\mathbb A_{\gg}$
by a $\GG$-invariant
scalar valued function on $\mathbb A_{\gg}$:
Let $\varphi\colon \mathbb R \to \mathbb R_{>0}$
be a smooth function of the (single) variable $x$ such that
the smooth function
\[
\chi\colon \mathbb R \longrightarrow \mathbb R,\ 
\chi(x)=x \varphi(x^2),
\]
is a
local diffeomorphism.
Notice that $\chi$ is then a diffeomorphism onto its image. 
We do not require that $\chi$ be onto.
Possible examples, beyond $\chi(x)=x$, are
$\chi(x)=\sinh(x)$ or $\chi(x)=\arctan(x)$.
Pick $\Phi$ such that
$
\Phi'= \tfrac 12 \varphi
$
and let  $\Xi(x)=\Phi(x^2)$ and 
\[
f(x)=x \chi(x)-\Xi(x)
=x^2 \varphi(x^2)-\Xi(x)
=x^2 \varphi(x^2)-\Phi(x^2);
\]
then
\begin{align*}
\Xi'(x)&=2x\Phi'(x^2)=x \varphi(x^2) =\chi(x).
\\
\chi'(x)&=2x^2\phi'(x^2)+\phi(x^2).
\\
f'(x)&=x \chi'(x).
\end{align*}
Define $\chi_{\gg}\colon \mathbb A_{\gg} \to \mathbb A_{\gg}$
by
\begin{equation}
\chi_{\gg}(\qqq)= \varphi(||\qqq||^2)\qqq.
\label{kind1}
\end{equation}
Then $[\chi_{\gg}(\qqq),\qqq]$
is zero for every $\qqq \in \gg$, whence
the 1-form $\langle \mu,\ppsi_{\chi}\rangle$ is zero.
Moreover
\begin{align*}
(d\chi_{\gg})_{\qqq}(V)&=2 \varphi'(||\qqq||^2) (\qqq \cdot V)\, \qqq 
+\varphi(||\qqq||^2)V
\\
\langle \mu,\llambda_{\chi}\rangle_\qqq(V)
&=
a\cdot ((d\chi_{\gg})_\qqq(V))
\\
&=
\left(2||\qqq||^2 \varphi'(||\qqq||^2)
+\varphi(||\qqq||^2)\right) \qqq \cdot V 
\\
&=
\chi'(||\qqq||)\, \qqq \cdot V .
\end{align*}
Now, define $F\colon \mathbb A_{\gg} \to \mathbb R$ by
\[
F(\qqq)=||\qqq||^2 \varphi(||\qqq||^2)-\Phi(||\qqq||^2)
=
||\qqq|| \chi(||\qqq||)-\Xi(||\qqq||),\ \qqq \in \mathbb A_{\gg}.
\]
Then
$dF=\langle \mu,\llambda_{\chi}\rangle$
whence, in particular, $\langle \mu,\llambda_{\chi}\rangle$ is closed.
The resulting complex structure $J$ on $\GG \times \mathbb A_{\gg}$
is biinvariant and combines with the symplectic structure
$\oo\, (=-d\langle \mu, \theta_{\GG}\rangle)$, the tautological structure
with respect to the chosen innner product on $\gg$,
to a biinvariant K\"ahler structure.

As a consistency check we note that,
in the special case where the function $\varphi$ has constant value $1$,
\begin{align*}
\chi(x)= x,\ 
\Xi(x)= \tfrac 12 x^2,\ 
f(x)= x^2-\tfrac 12 x^2=\tfrac 12 x^2,
\
F(\qqq)&= \tfrac 12 \qqq \cdot \qqq .
\end{align*}
These are the corresponding identities in the standard case,
cf. \eqref{funct1}.

For example, with $\varphi(x^2)=\tfrac{\arctan(x)}x$,
the resulting map $\chi_{\gg}\colon \mathbb A_{\gg} \to \mathbb A_{\gg}$
given by $\chi_{\gg}(\qqq)=\varphi(||\qqq||^2)\, \qqq$ ($\qqq \in \mathbb A_{\gg}$)
is a  diffeomorphism onto its image but is not onto.
The resulting generalized polar map $\Pi\colon \GG \times \mathbb A_{\gg}
\to \GG^{\mathbb C}$ is then a biinvariant
K\"ahler diffeomorphism
onto a proper open subset of
$\GG^{\mathbb C}$ (endowed with the standard structure); 
in particular,  $\Pi$ is not onto.

More generally, 
we can rescale the identity of $\mathbb A_{\gg}$
with a smooth function
$\varphi(i_1(\cdot),\ldots,i_{\ell}(\cdot))$
of the invariants $i_1,\ldots,i_{\ell}$ 
of $\gg$, that is,
consider a local diffeomorphism
 $\chi_{\gg}\colon \mathbb A_{\gg} \to \mathbb A_{\gg}$
of the kind  
\begin{equation}
\chi_{\gg}(\qqq)=\varphi(i_1(\qqq),\ldots,i_{\ell}(\qqq)) \qqq,\ 
\qqq \in \mathbb A_{\gg}.
\label{kind2}
\end{equation}

This class of examples raises the following question:
Suppose that $G$ is simple.
Are there biinvariant K\"ahler structures on $\GG \times \mathbb A_{\gg}$
having underlying symplectic structure the tautological one
(relative to an invariant inner product on $\gg$)
that are distinct from those arising from
rescaling the identity?

\section*{Acknowledgement}
The authors are indebted to R. Bielawski, for his paper \cite{MR1989647},
which was a great source of inspiration,
as well as for some email correspondence. 
The authors also profited from email correspondence with R. Sz\"oke.
Support by the CNRS and by the
Labex CEMPI (ANR-11-LABX-0007-01)
is gratefully acknowledged.

\addcontentsline{toc}{section}{References}
\def\cprime{$'$} \def\cprime{$'$} \def\dbar{\leavevmode\hbox to 0pt{\hskip.2ex
  \accent"16\hss}d} \def\cprime{$'$} \def\cprime{$'$} \def\cprime{$'$}
  \def\cprime{$'$} \def\cprime{$'$}
  \def\polhk#1{\setbox0=\hbox{#1}{\ooalign{\hidewidth
  \lower1.5ex\hbox{`}\hidewidth\crcr\unhbox0}}}

\end{document}